\documentclass[a4paper,10.95pt]{amsart}
\usepackage[utf8]{inputenc} \usepackage[T1]{fontenc}
\usepackage{lmodern}
\usepackage{amsmath, mathtools, amssymb, amsthm, amsbsy, bbm, amsfonts, enumitem, color, url}
\usepackage{float}
\usepackage[all]{xy}
\usepackage[english]{babel}
\usepackage{graphicx, color}
\usepackage{epic}
\usepackage[dvipsnames]{xcolor}
\usepackage[most]{tcolorbox}
\usepackage[pdfpagelabels]{hyperref}
\usepackage[Option]{overpic}
\usepackage[capitalise]{cleveref}

\usepackage[normalem]{ulem}
\usepackage{caption}
\definecolor{ourgreen}{rgb}{0, 0.6, 0.5}
\definecolor{ourred}{rgb}{0.8, 0.4, 0}
\definecolor{ourblue}{rgb}{0, 0.45, 0.7}

\newtheoremstyle{redthm}
  {}
  {}
  {\color{Black}\itshape\sffamily}
  {}
  {\bfseries\color{Red}\sffamily}
  {.}
  {5pt}
  {}

\newtheoremstyle{ques}
  {}
  {}
  {\color{Black}\itshape\sffamily}
  {}
  {\bfseries\color{Blue}\sffamily}
  {.}
  {5pt}
  {}
\newtheoremstyle{greendef}
  {}
  {}
  {\sffamily\color{Blue}}
  {}
  {\bfseries\color{ForestGreen}\sffamily}
  {.}
  {5pt}
  {}

\theoremstyle{redthm}
\newtheorem{thm}{Theorem}[section]

\newtheorem*{thm*}{Theorem}
\newtheorem{prop}[thm]{Proposition}
\newtheorem{lemma}[thm]{Lemma}
\newtheorem{corollary}[thm]{Corollary}

\newtheorem{Ass}[thm]{Assumptinos}

\theoremstyle{ques}
\newtheorem{ques}[thm]{Question}
\newtheorem*{ques*}{Question}

\theoremstyle{greendef}
\newtheorem{definition}[thm]{Definition} 
\newtheorem{example}[thm]{Example}

\theoremstyle{remark}
\newtheorem{remark}[thm]{Remark}

\theoremstyle{observation}

\theoremstyle{plain}

\numberwithin{equation}{thm}

\newcommand{\R}{\mathbb{R}}

\newcommand{\mc}{\mathcal}
\newcommand{\N}{\mathbb{N}}
\newcommand{\cay}{\mathop{\mathsf{Cay}}}

\newcommand{\emb}{\hookrightarrow}

\newcommand{\hull}[2]{\mathrm{a\text{-}Core}_{#2}(#1)} 
\newcommand{\core}[2]{\mathrm{Core}_{#2}(#1)} 
\newcommand{\comp}[2]{\mathrm{Comp}_{#2}(#1)} 

\newcommand{\abs}[1]{\left\lvert #1 \right\rvert}

\newcommand{\mb}{\partial_*}

\newcommand{\acts}{\curvearrowright}
\newcommand{\ra}{\rightarrow}

\newcommand{\cu}{\subseteq}

\newcommand{\mf}{\mathfrak}

\newcommand{\s}{\sigma}
\newcommand{\eps}{\epsilon}

\newcommand{\g}{\gamma}
\newcommand{\G}{X}

\DeclareMathOperator{\diam}{diam}

\newcommand{\E}{\mc{E}}
\renewcommand{\H}{\mc{H}}

\title{Subgroups arising from connected components in the Morse boundary 
} 

\author {Annette Karrer, Babak Miraftab, Stefanie Zbinden}

\begin{document}

\begin{abstract}
We study connected components of the Morse boundary and their stabilisers. We introduce the notion of \textit{point-convergence} and show that if the set of non-singleton connected components of the Morse boundary of a finitely generated group $G$ is point-convergent, then every non-singleton connected component is the (relative) Morse boundary of its stabiliser. The above property only depends on the topology of the Morse boundary and hence is invariant under quasi-isometry. This shows that the topology of the Morse boundary not only carries algebraic information but can be used to detect certain subgroups which in some sense are invariant under quasi-isometry. 
\end{abstract}

\maketitle

\vspace{- 1.0cm}
\section{Introduction}

 The \textit{Morse boundary} of a space captures its hyperbolic-like behaviour at infinity. It was introduced by Cordes~\cite{Cordes} generalising the contracting boundary of Charney--Sultan~\cite{CHAR14}. As it is a quasi-isometry invariant, it can be used to distinguish finitely generated groups up to quasi-isometry, as demonstrated in \cite{CHAR14} and \cite{zbinden2023small}.
 
 In a different direction, we propose to investigate what algebraic information the Morse boundary carries. For example, information about subgroups and splittings. This seems likely, since Morse boundaries generalise Gromov boundaries, which are well known in this regard (see e.g.~\cite{Bowditch_top_char, Bowditch_cutpairs,Swenson,Mahan}). The known splitting results on the Bowditch boundary~\cite{Bowditch_splitting},~\cite{Haulmark},~\cite{Haulmark-Hruska} the combination results about the Morse boundary~\cite{FK22},~\cite{zbinden2023morse} and the capability to detect stable subgroups based on their action on the Morse boundary \cite{CD16}, further support this proposal.

In this paper, we aim to enforce this idea by showing that connected components of the Morse boundary (later referred to as boundary components) indeed carry algebraic information about subgroups. If $C$ is a non-singleton connected component in the Gromov boundary of a hyperbolic group, then $C$ is homeomorphic to the Gromov boundary of its stabiliser as a consequence of the accessibility of hyperbolic groups~\cite{Gromov, Dunwoody}.
We generalise this result to the Morse boundary under the assumption that the orbit of the boundary component under consideration satisfies a discreteness property we call \emph{point-convergence} (see \cref{definition:point-convergence} and the discussion below). 

\begin{thm}\label{thm:main}
   Let $G$ be a finitely generated group. Let $C$ be a non-singleton connected component of the Morse boundary $\mb G$. If $G\cdot C$ is point-convergent, then $(\mb H,  G) = C$, where $H\leq G$ is the stabiliser of $C$ under the action $G \acts \mb G$.
\end{thm}
Here, $(\mb H ,G)$ is the relative Morse boundary of $H$ in $G$ and consists of all equivalence classes of Morse geodesic rays in $G$ which are at bounded distance from~$H$. 

We prove that the orbit of every boundary component is point-convergent for groups that split in a special way or are quasi-isometric to such groups (see \cref{cor:graph_of_groups} and \cref{cor:q_i_to_amalgam}). To our knowledge, there is no known example of a group with a boundary component whose orbit is not point-convergent.

\textit{Point-convergence} (\cref{definition:point-convergence}) generalises the notion of \textit{being a null family}. In particular, for compact metric spaces - for instance the Gromov boundary of hyperbolic groups -  both notions coincide. Null families play an important role in the study of boundaries~\cite{width},~\cite{Mahan},~\cite{Hruska_Ruane}, \cite{dasgupta2022local},\cite{benzvi2019boundaries} and decompositions of manifolds~\cite{Daverman}[VI].
They were in particular used to characterise quasi-convex subgroups of hyperbolic groups, partially by Gitik--Mitra--Rips--Sageev~\cite{width}[Cor 2.5] and implicitly but fully by Mj~\cite{Mahan}[Prop 2.3]. This and \cref{thm:main} suggest  that point-convergence is a notion of interest and could be further useful in the study of Morse boundaries and other non-compact non-metrizable spaces. 

If the family of all non-singleton boundary components of a group is point-convergent, so is any particular orbit of a boundary component. Since the former is a purely topological property of the Morse boundary (and hence invariant under quasi-isometries) \cref{thm:main} has the following consequence. 
\begin{corollary}\label{cor:q_i_toG'}
  Let $G$ be a finitely generated group for which the family of all non-singleton connected components of the Morse boundary $\mb G$ is point-convergent. Then for any group $G'$ which is quasi-isometric to $G$, we have that $(\mb H',G') = C'$ for every non-singleton connected component $C'\subset \mb G'$ and its stabiliser $H'\leq G'$.
\end{corollary}

\cref{cor:q_i_toG'} shows that the Morse boundary (at least for ``nice'' groups) carries information about the subgroup structure which is preserved under quasi-isometry. This is remarkable because in general, quasi-isometries do not preserve the subgroup structure of groups at all. A natural question is whether the following strengthening of \cref{cor:q_i_toG'} is true. 

\begin{ques}\label{ques:strong_qi}
 Let $f :G\to G'$ be a quasi-isometry between finitely generated groups $G$ and $G'$. Assume that the family of all non-singleton connected components of the Morse boundary $\mb G$ is point-convergent and let $C$ be a non-singleton connected component of the Morse boundary $\mb G$. Is the stabiliser of $f(C)$ quasi-isometric to the stabiliser of $C$?
\end{ques}

A positive answer to \cref{ques:strong_qi} would be interesting, as it would show an even stronger correlation between the Morse boundary and algebraic properties of the group.

It turns out to be quite involved to prove or disprove point-convergence. As mentioned before, we do not know an example of a group with a boundary component whose orbit is not point-convergent.

\begin{ques}\label{ques}
    Let $G$ be a finitely generated group. Is the family of non-singleton connected components of its Morse boundary $\mb G$ point-convergent? If $G$ is a group hyperbolic relative to abelian groups, are the non-singleton connected components of its Morse boundary point-convergent?
\end{ques}

For a certain class of groups that split ``nicely'', we use the results from \cite{FK22} to answer \cref{ques} positively.

\begin{thm}\label{cor:graph_of_groups}
Let $G$ be a finitely generated group which splits as a graph of groups. Suppose that
\begin{itemize}
    \item All edge groups are finitely generated and undistorted in $G$;
    \item $(\partial_*E,G)=\emptyset$ for every edge group $E$.
     \item For every vertex groups $V$, either $(\partial_*V,G)$ is a connected subspace of $\partial_* G$ or $V$ satisfies the following two items:
   \begin{itemize}
        \item The inclusion $\partial_*V \emb \partial_* G$ is a well-defined topological embedding. In particular, $\mb V \cong (\mb V, G)$. 
        \item The family of non-singleton connected components of $\partial_* V$ is point-convergent.
    \end{itemize}
\end{itemize}
Then the family of non-singleton connected components of $\partial_*G$ is point-convergent.
\end{thm}

There are many groups which satisfy the assumption of the theorem above. For instance, all fundamental groups of oriented closed 3-manifolds (an explanation of this can be found in the proof of Theorem 7.2 of \cite{zbinden2022morse}). Moreover, gluing hyperbolic 3-manifolds with at least two cusps along one of their boundary tori gives further examples, since Charney-Cordes-Sisto show in \cite{charney2020complete} that the Morse boundary of a hyperbolic 3-manifold with cusps is connected.
 
The following result is an application of \cref{cor:q_i_toG'} and \cref{cor:graph_of_groups}.

\begin{corollary}\label{cor:q_i_to_amalgam}
    Let $G'$ be a finitely generated group quasi-isometric to a group $G$ as in \cref{cor:graph_of_groups}. Then for every non-singleton connected component $C'\subset \mb G'$ we have that $(\mb H',G') = C'$, where $H'\leq G'$ is the stabiliser of $C'$ in $ G' \acts \mb G'$.
\end{corollary}

Again, one can ask whether \cref{cor:q_i_to_amalgam} can be strengthened. Namely, one can ask whether not only the stabiliser of connected components structure but also the splitting structure is preserved under quasi-isometries. To formulate this question, we define \textit{Morse-accessibility}.

\begin{definition}\label{def_Morse_accessible}
    We say that a finitely generated group $G$ is \textit{Morse-accessible} if a process of iterated nontrivial splittings of $G$ along groups with empty relative Morse boundary terminates after finitely many steps and all vertex groups of the last iteration have connected Morse boundary. 
\end{definition}

For instance, the fundamental group of an oriented closed 3-manifold is Morse-accessible (again, an explanation of this can be found in the proof of Theorem 7.2 of \cite{zbinden2022morse}).

\begin{ques}\label{ques_Morse_acc}
    Let $G'$ be a finitely generated group that is quasi-isometric to a finitely generated Morse-accessible group $G$. Is $G'$ Morse-accessible as well?
\end{ques}

For hyperbolic groups, the answer to this question is positive as every hyperbolic group is accessible~\cite{Dunwoody,Gromov}.

\subsection*{Outline} In \cref{sec:prelim}, we recall some background about the Morse boundary and introduce the definition of relative Morse boundary, which is different from the definition used in other papers. In \cref{sec:point:conv} we define point-convergence and show that it can and should be viewed  as a generalisation of the well-known notion of null-family for non-metrizable spaces. Given a metric space $X$, we introduce the notion of core of a boundary component, which is a subset of its weak (convex) hull. The $N$--core of $X$, given a Morse gauge $N$ and a vertex in $X$, is a subspace of $X$ associated to $C$. It can be thought of as the ``inner'' of $C$ with respect to a certain Morse gauge $N$. In some sense dual to the notion of core, we introduce the notion of anti-core, which is a set of connected components of $\mb X$ which is associated to a given point of $X$. As the core, the anti-core depends on a choice of the Morse gauge. Lastly, in \cref{lemma:orbit-convergence_implies_finite_hull} we relate having finite anti-core to certain families of connected components of $\mb X$ being point-convergent. In \cref{sec:stabiliser} we prove \cref{thm:main}. The main tools for the proof are the notions of core and anti-core as defined in the previous section and \cref{lemma:orbit-convergence_implies_finite_hull}. Lastly, in \cref{sec:graph_of_groups}, we prove \cref{cor:graph_of_groups}. This is done by using the tools from \cite{FK22}.

\subsection*{Acknowledgements} We would like to thank Alessandro Sisto, Ruth Charney, Merlin Incerti-Medici, Jean-François Lafont, Matthew Cordes, Matthew R. Haulmark, Vivian He, Mahan Mj, and Alex Margolis for the helpful discussions and the great support we received. Karrer was partially supported by CIRGET and CRM, by the NSF (under grant DMS-2109683), and the fields institute within the Thematic Program on Randomness and Geometry.

\section{Preliminaries and Notation}\label{sec:prelim}

In this section, unless noted otherwise, $(X,d)$ denotes a proper geodesic metric space.

\subsection{Morseness}

\begin{definition}[Morse gauge] A function $M : \R_{\geq1}\times \R_{\geq 0}\to \R_{\geq 0}$ is a Morse gauge if it is continuous in the second coordinate and increasing. We denote the set of all Morse gauges by $\mc M$.
\end{definition}

\begin{definition}[Morse] Let $M$ be a Morse gauge. A quasi-geodesic $\gamma$ is $M$-Morse if all $(q, Q)$-quasi-geodesics with endpoints $\gamma(s)$ and $\gamma(t)$ for $s\leq t$ stay in the closed $M(q, Q)$-neighbourhood of $\gamma[s, t]$. A quasi-geodesic is Morse if it is $M$-Morse for some $M$.
\end{definition}

\begin{remark}
    This definition deviates slightly from the definition of Morse as defined in \cite{Cordes}. However, it is equivalent in the sense that if a geodesic is $M$-Morse for one of the definitions, it is $M'$-Morse for the other, where $M'$ only depends on $M$ (see \cite[Lemma 2.1]{Cordes} and \cite[Lemma A.4]{CSZ23}). With our definition, any subgeodesic of an $N$-Morse geodesic is $N$-Morse. 
\end{remark}

If a geodesic is Morse, we think of it as exhibiting hyperbolic-like properties. To every Morse gauge $M$ one can associate a constant $\delta_M = \max\{4M(1, 2M(5, 0)) +
2M(5, 0), 8M(3, 0)\}$. The constant $\delta_M$ should be thought of as the hyperbolicity constant of $M$-Morse geodesics.

\textbf{Notation:} Let $x, y\in X$. We say that $[x, y]$ is $M$-Morse if there exists an $M$-Morse geodesic from $x$ to $y$.

\subsection{The Morse boundary}

Let $(X, d)$ be a proper geodesic metric space with basepoint $e$. The $N$-Morse boundary based at $e$, denoted by $\partial_e^N X$, is defined as 
\begin{align*}
     \partial_e^N X = \{\gamma | \text{$\gamma\colon [0, \infty) \to X$ is an $N$-Morse geodesic ray starting at $e$} \}_{/\sim},
\end{align*}

where $\gamma\sim\gamma'$ if they have bounded Hausdorff distance. We call an equivalence class $[\gamma]\in \partial_e^N X$ a Morse direction. For any Morse quasi-geodesic ray $\gamma$ not necessarily starting at $e$, there exists a Morse geodesic ray $\gamma'$ starting at $e$ which is at bounded Hausdorff distance of $\gamma$. In this case we say that $\gamma\sim \gamma'$ and define $[\gamma]$ as $[\gamma']$. If there exists such a ray $\gamma'$ which is $N$-Morse (even if $\gamma$ is not $N$-Morse), we will say that $[\gamma]\in \partial_e^N X$.

\textbf{Notation.} If $\gamma\colon [0, \infty)\to X$ is a Morse geodesic ray, we denote \emph{its endpoint at infinity} $[\gamma]$ by $\gamma^+$. Similarly, if $\gamma \colon \R \to X$ is a Morse geodesic line, we denote its \emph{endpoints} $[\gamma[0, -\infty)]$ and $[\gamma[0, \infty)]$ by $\gamma^-$ and $\gamma^+$ respectively. 

The topology of the $N$-Morse boundary based at $e$ is given by the neighbourhood bases $U_k^N(\gamma)_e = \{[\gamma']| \text{$d(\gamma(t),\gamma'(t))\leq \delta_N$ for all $t\leq k$}\}$. Where $\delta_N $ is the hyperbolicity constant of $N$ as defined above.

In \cite{CH17} Cordes and Hume prove that the $N$-Morse boundary is compact. 

\begin{lemma}[\cite{CH17}]\label{lemma:strata_are_compact}
    Let $N$ be a Morse gauge, then $\partial_e^N X$ is compact. 
\end{lemma}

The Morse boundary of $(X, e)$ is the direct limit 
\begin{align}
    \partial_e X = \lim_{\xrightarrow[\mathcal{M}]{}} \partial_e^M{X},
\end{align}
which is endowed with the direct limit topology. 
As a topological space, the Morse boundary does not depend on the basepoint $e$, and thus, when the basepoint is irrelevant, we denote it by $\mb X$.

In \cite{CD16} the converse to \cref{lemma:strata_are_compact} is proven.

\begin{lemma}[Lemma 4.1 of \cite{CD16}]\label{lemma:compact_subsets}
    Let $K\subset \mb X$ be compact, then there exists a Morse gauge $N$ such that $K\subset \partial_e^N X$.
\end{lemma}

 A triangle is \emph{$\delta$-slim} if each side is contained in the $\delta$-neighbourhood of the
other two sides. The following lemma states that Morse triangles are slim. A proof for the finite case can be found in \cite{Cordes} and in the general case in \cite{CCM19}.

\begin{lemma}[Morse triangles, Lemma 2.3 of \cite{CCM19}]\label{lemma:triangles}
There exists a function $f \colon \mc M\to \mc M$ such that the following holds. Let $(\alpha, \beta, \gamma)$ be a geodesic triangle with vertices in $\mb X \cup X$. If the sides $\alpha$ and $\beta$ are $M$-Morse, then $\gamma$ is $f(M)$-Morse and contained in the $\delta_M$-neighbourhood of $\alpha\cup\beta$.
\end{lemma}

\begin{lemma}[Lemma A.6 of \cite{CSZ23}]\label{lemma:convergence_of_N_morse} Let $N$ be a Morse gauge and let $(\gamma_n)$ be a sequence of $N$-Morse geodesic rays starting at a point $x$ and converging (uniformly on compact sets) to $\gamma$. Then $\gamma$ is an $N$-Morse geodesic ray.
\end{lemma}

Note that not only do the rays $(\gamma_n)$ converge to $\gamma$, but also the Morse directions $(\gamma_n^+)$ converge to $\gamma^+$ in the Morse boundary. 

The following is part of Lemma 2.8 in \cite{zbinden2023morse} summarising basic properties of Morse rays.
\begin{lemma}[ Lemma 2.8 (ix) in \cite{zbinden2023morse} ]\label{lem_neibMorse}
 Let $\gamma$ be an M-Morse quasi-geodesic and $\gamma'$ be a finite $C$-quasi-geodesic segment with endpoints in the $C$-neighbourhood of $\gamma$. Then $\gamma'$ is $M'$-Morse where $M'$ depends only on $M$ and $C$. 

\end{lemma}
\subsection{Morseness and groups}

For the rest of this paper, unless noted otherwise $G$ denotes a finitely generated group with finite generating set $S$, $X = \cay(G, S)$ denotes its Cayley graph and $d$ denotes both the path-metric on $X$ and the induced word metric on $G$. 
We note that $G$ is the vertex set of $X$.

The group $G$ acts by left-multiplication on $X$. This induces an action on (quasi-) geodesics of $X$ and hence it extends to an action of $G$ on the Morse boundary $\mb X$. 
Some properties of this action are: 

\begin{itemize}
    \item If $\gamma^+\in \partial_e^N X$, then $g\cdot \gamma^+\in \partial_{g\cdot e}^NX$. However, in general, $g\cdot \gamma^+\not \in \partial_e^N X$.
    \item If $C\subset \mb X$ is connected, then $g\cdot C$ is connected. In particular, if $C\subset \mb X$ is a connected component of $\mb X$, then $g\cdot C$ is a connected component of $\mb X$.
\end{itemize}

Given a point $v\in X$, an $M$-Morse geodesic line $\lambda$ and an element $g\in G$, one might want to construct a geodesic ray $\lambda_{gv}$ which starts at 
$v$, passes close to $gv$ and is $M'$-Morse, where $M'$ depends only on $M$. The following lemma states that such a ray can always be found if $[v, gv]$ is $M$-Morse. Indeed, as proven in \cite[Lemma 2.27]{zbinden2023morse}, one of the rays starting at $v$ and ending in $g\cdot \lambda^+$ or $g\cdot \lambda^-$ satisfies these properties. 
\begin{figure}[H]
    \centering
    \includegraphics[scale=0.7]{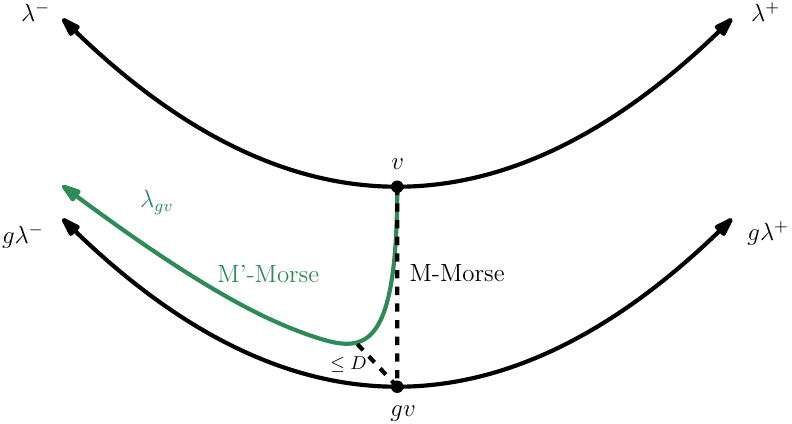}
    \caption{Definition of $\lambda_{gv}$.}
    \label{fig:enter-label}
\end{figure}
\begin{lemma}[Version of Lemma 2.27 of \cite{zbinden2023morse}]\label{lemma:corresponding_ray} For every Morse gauge $M$, there exists a Morse gauge $M'$ and a constant $D$ such that the following holds. Let $v\in X$ be any point and let $\lambda$ be an $M$-Morse geodesic line going through $v$. For every $g\in G$ for which $[v, gv]$ is $M$-Morse, there exists an $M'$-Morse geodesic ray $\lambda_{gv}$ starting at $v$ and satisfying the following properties.
\begin{enumerate}[label = \roman*)]
\item $\lambda_{gv}^+\in \{g\cdot \lambda^-, g\cdot\lambda^+\}$\label{prop:boundary}
\item $d(gv, \lambda_{gv})\leq D$ 
\label{prop:close_to_corresponding_ray}
\end{enumerate}
\end{lemma}

\begin{proof}
    Property \ref{prop:boundary}, follows directly from the construction of $\lambda_{gv}$. Using the notation of \cite[Lemma 2.27]{zbinden2023morse}, we can choose $M' = M_\lambda$ and get that $\lambda_{gv}$ is $M'$-Morse. Property \ref{prop:close_to_corresponding_ray} is proven in the proof of \cite[Lemma 2.27]{zbinden2023morse} for $D = \delta_{M_\lambda}$.
\end{proof}

\subsection{Relative Morse boundary} 

Sometimes we consider a pair of spaces $Y\subset X$ and are interested in the Morse boundary of $Y$ viewed as a subset of $X$. We call this the relative Morse boundary.

\begin{definition}[Relative Morse boundary]\label{definition:relative_morse_boundary}
Let $Y$ be a subset of $X$. 
We say that the relative Morse boundary of $Y$ (with respect to $X$), denoted by $(\mb Y, X)$, is the set containing all Morse directions $\gamma^+\in \mb X$ such that there exists $k$ for which $\gamma$ is contained in the $k$-neighbourhood of $Y$. We endow $(\mb Y, X)\subset \mb X$ with the subspace topology.

\textbf{This is different than the definition from \cite{FK22}, however, we show that for undistorted subgroups $H\leq G$, the above definition agrees as a set with the definition from \cite{FK22}}.
\end{definition}

\begin{definition}\label{def:undistorted}
    A subgroup $H$ of a finitely generated group $G$ is undistorted if it is finitely generated and the inclusion $H\to G$ is a quasi-isometric embedding.
\end{definition}

The following lemma shows that as a set, our definition of relative Morse boundary agrees with the definition from \cite{FK22} for undistorted subgroups $H\leq G$.

\begin{lemma}\label{lemma:defintions_agree}
    Let $(X, d_X)$ and $(Y, d_Y)$ be proper geodesic metric spaces and let $\iota \colon Y\to X$ be a quasi-isometric embedding. Let $\gamma \colon [0, \infty) \to X$ be a Morse geodesic ray. The following are equivalent: 
    \begin{enumerate}
        \item $\gamma^+\in (\mb \iota(Y), X)$
        \item there exists a geodesic ray $\lambda\colon [0, \infty)\to Y$ which is Morse in $Y$ and whose image $\iota\circ\lambda$ is at bounded Hausdorff distance from $\gamma$
    \end{enumerate}
\end{lemma}
\begin{proof}
    (2) $\implies$ (1): If $\gamma$ is at bounded Hausdorff distance of $\iota\circ\lambda$, then $\gamma$ is at bounded Hausdorff distance from $\iota (Y)$ and hence $\gamma^+\in (\mb \iota(Y), X)$.

    (1) $\implies$ (2): Assume that $\gamma^+\in(\mb \iota(Y), X)$. Hence $\gamma$ is in the $k$-neighbourhood of $\iota(Y)$ for some $k>0$. For all $t\geq 0$, let $y_t\in Y$ be a point such that $d(\iota(y_t), \gamma(t))\leq k$. Consider $\lambda: [0, \infty)\to Y$ defined by $\lambda(t) = y_{t}$. Since $\iota$ is a quasi-isometric embedding and $\gamma$ is a Morse geodesic satisfying $d(\gamma(t),  \iota \circ \lambda(t))\leq k$ for all $t\geq 0$, both $\iota\circ \lambda$ and $\lambda$ are indeed Morse quasi-geodesic rays.
\end{proof}

\section{Point-convergence and its characterisations}\label{sec:point:conv}
In this section, we introduce the notion of point-convergence, which is a notion that quantifies the ``discreteness of families of subsets'' of a topological space. We then show that point-convergence can be viewed as a generalisation of the notion of a null-family, which is a well-studied notion with respect to the Gromov boundary.

In Subsection \ref{sec:hullandCore}, for a group $G$, we introduce the concept of $N$--anti-core and $N$--core which are sets of connected components of the Morse boundary $\mb G$ and a subset of the Cayley graph $X$ respectively. We then relate having finite $N$--anti-core to the set of connected components of $\mb X$ being point-convergent.

\subsection{Null-families and point-convergence}\label{sec:nullfamily}

We denote the cardinality of a set $A$ by $|A|$.
\begin{definition}\label{def:nullfamily}
    Let $Y$ be a compact metric space. A family of subsets $(A_i)_{i\in I}$ of $Y$ is called a \emph{null-family} if for all $\eps > 0$ 
    \begin{align*}
        \abs{\{i\in I | \diam(A_i) > \eps\}}< \infty.
    \end{align*}
\end{definition}

We want to show that if the set of connected components in the Morse boundary is a null-family, then the stabilisers of the connected components behave nicely. However, the Morse boundary is neither compact nor metrizable in general, so we introduce the notion of point-convergence, which coincides with the notion of null-family for compact metric spaces but is defined in the more general setting of topological spaces.

So first we define what the limit of a sequence of subsets is. 

\begin{definition}
    Let $Y$ be a topological space and let $(A_n)_{n\in \N}$ be a sequence of subsets of $Y$. The limit 
    \begin{align*}
        L = \lim_{n\to \infty} A_n
    \end{align*}
    is the set of points $x\in Y$ which can be written as $x = \lim_{n\to \infty} x_n$ for points $x_n\in A_n$.
\end{definition}

We are now ready to define point-convergence.

\begin{definition}\label{definition:point-convergence}
    Let $Y$ be a topological space. A family of subsets $\{A_i\}_{i\in I}$ of $Y$ is \emph{point-convergent} if for all sequences $(A_{i_n})_{n\in \N}$ of distinct elements of $\{A_i\}_{i\in I}$ we have that
    \begin{align}
        \abs{\lim_{n\to \infty} A_{i_n}}\leq 1.
    \end{align}
\end{definition}

\begin{lemma}\label{lemma:null_to_convergence}
    Let $Y$ be a compact metrizable space and let $\{A_i\}_{i\in I}$ be a family of subsets of $Y$. The family $\{A_i\}_{i\in I}$ is a null-family if and only if it is point-convergent.
\end{lemma}
\begin{proof}
    ``$\Rightarrow$''  Assume that $\{A_i\}_{i\in I}$ is a null-family but not point-convergent. By assumption, there exists a sequence $(A_{i_n})_{n\in \N}$ of distinct elements of $\{A_i\}_{i\in I}$ whose limit $L$ has cardinality at least 2. Say $x\neq y\in L$. Let $\eps = d (x, y)$. Since $x$ and $y$ are in the limit, for large enough $n$, there exist points $x_n, y_n\in A_{i_n}$ with $d(x, x_n) < \eps/4$ and $d(y, y_n)< \eps /4$. By triangle inequality $d(x_n, y_n) > \eps/2$. Thus for all large enough $n$, $\diam (A_{i_n}) > \eps/2$, which is a contradiction to  $\{A_i\}_{i\in I}$ being a null-family.

  ``$\Leftarrow$'' Assume that $\{A_i\}_{i\in I}$ is point-convergent but not a null-family. By assumption, there exists $\eps > 0$ and a sequence $(A_{i_n})_{n\in \N}$ of distinct elements of $\{A_i\}_{i\in I}$ such that $\diam(A_{i_n})> \eps$ for all $n$. In other words, for all $n\in \N$ there exist points $x_n, y_n\in A_{i_n}$ with $d(x_n, y_n) >  \eps$. Since $Y$ is compact, up to passing to a subsequence, we can assume that the sequences $(x_n)_n$ and $(y_n)_n$ converge to points $x$ and $y$ respectively. Since $d(x_n, y_n) > \eps$ for all $n$, we have that $d(x, y)\geq \eps$ and most importantly $x\neq y$. By definition, $x$ and $y$ are both in the limit of $(A_{i_n})_{n\in \N}$, a contradiction to point-convergence.
\end{proof}

 The following example illustrates a hyperbolic space (but not group!) $X$ whose set of non-singleton boundary components fail to build a null-family or equivalently to be point-convergent and shows how this failure is related to the fact that the space $X$ is not accessible: 
 \begin{example}\label{example1}
     Let $X$ be a space obtained by gluing a sequence of hyperbolic planes $(H_i)_{i\in\mathbb{N}}$ with base points $(b_i)_{i\in \N}$ on a hyperbolic plane $H_{\infty}$ with basepoint $b$ as follows: for every $i\in\N$, we identify the disk with midpoint $b_i$ and radius $i$ in $H_i$ and the disk with midpoint $b$ and radius $i$ in $H_{\infty}$, as illustrated in Figure~\ref{H}. We denote by $C_i$ the Gromov boundary of $H_i$, which is a circle. We note that $X$ is a proper Gromov-hyperbolic space, and its Gromov boundary is compact and is the union of the circles $C_i$ and $C_{\infty}$. This is depicted in Figure~\ref{H}. 
    \begin{figure}[H]
        \centering
        \includegraphics[scale=0.7]{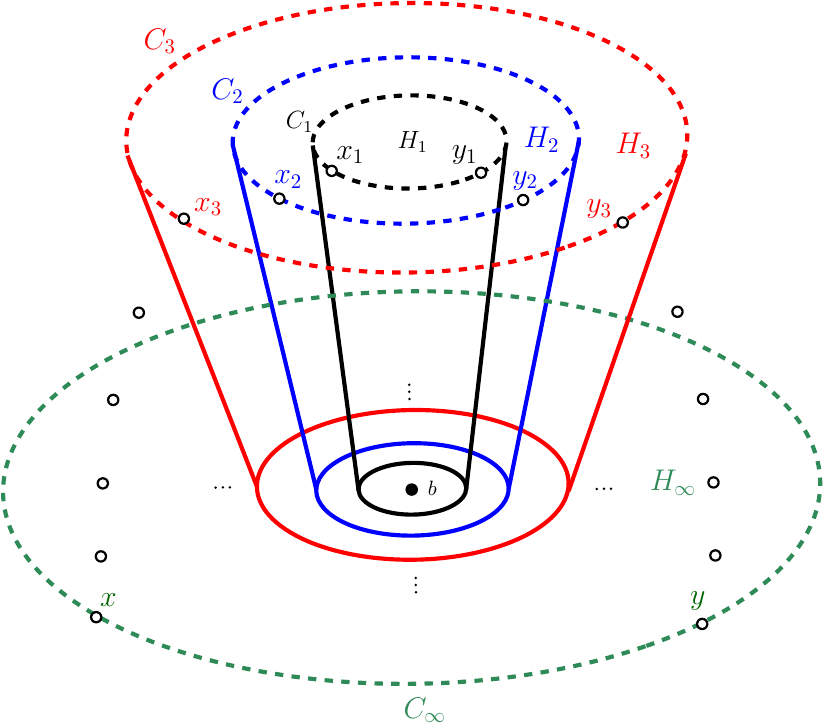}
        \caption{A hyperbolic space whose set of boundary components is not point-convergent.}
        \label{H}
    \end{figure}
  
If we equip the Gromov boundary of $X$ with a visual metric, we see that the diameter of all the circles $C_i$ are identical. As there are infinitely many such circles, the sequence $(C_i)_{i \in \mathbb{N}}$ does not form a null-family. Equivalently, $(C_i)_{i \in \mathbb{N}}$ is not point-convergent. Indeed, the sequences $(x_i)_{i\in \N}$ and $(y_i)_{i\in \N}$ as depicted in Figure~\ref{H} converge to two distinct points $x,y\in C_{\infty}$. If $(C_i)_{i \in \mathbb{N}}$ were point-convergent, the sequences $(x_i)_{i \in \mathbb{N}}$ and $(y_i)_{i \in \mathbb{N}}$ depicted in the figure would converge to a single point rather than to the two pictured distinct points $x$ and $y$.\\
Note that $X$ is not accessible, that is, there does not exist $D>0$ such that every pair of ends of $X$ can be separated by deleting a ball of radius $D$.
 \end{example}

   There is another generalisation of the notion of null-family which was for instance used in \cite{dasgupta2022local} but which is not equivalent to point-convergence: A family of subsets $\mathcal C$ of a topological space $Y$ is \textit{null} if for each open cover $\mc U$ of $Y$, all but finitely many members of $\mc C$ are $\mc U$-small. A set is $\mc U-$small if it is contained in some $U \in \mc U$. 

   \begin{example}
       Consider $Y = (0,  \infty)$. For all $n\in \N$ define $C_n = (n, n+\frac{1}{n})$. We claim that the set $\mc C = \{C_n\}_{n\in \N}$ is point-convergent (this follows since the $C_n$ are intervals whose length tends to $0$) but not a ``null-family'' as used in \cite{dasgupta2022local}. Indeed, consider the open cover $\mc U = \{(x, x+\frac{1}{n^2}) | n\in \N, x\in (n-1, n]\}$. For $n\geq 2$, none of the sets $C_n$ is $\mc U$--small.
   \end{example}

\subsection{The anti-core and the core of a non-singleton component}\label{sec:hullandCore}

Given a connected component $C$, one might want to study the \textit{weak (convex) hull} of $C$. The \textit{weak (convex) hull of a subset} of $\mb G$ consists of the union of all lines with endpoints in the subset. It was in particular used by Swenson~\cite{Swenson} and Cordes--Durham~\cite{CD16} to characterise quasi-convexity respectively stability of subgroups by their actions on the weak convex hull of their limit sets. Here, we introduce the notion of the core, which is a stable subset of the weak hull. Further, we introduce the notion of the anti-core, which is a set of connected components and can be viewed as a dual of the core. 

\begin{definition}[$N$--feasible] \label{def:feasible} Let $v\in X$ be a vertex and let $C$ be a non-singleton connected component of $\mb X$. We say that the pair $(C, v)$ is \emph{$N$--feasible} if there exists an $N$-Morse geodesic line $\lambda$ such that $\lambda(0) = v$ and $\lambda^+, \lambda ^-\in C$. We call $\lambda$ a witness of the pair $(C, v)$. 

We denote by $\comp{v}{N}$ the set of connected components $C$ of $\mb X$ such that $(C, v)$ is $N$--feasible.
\end{definition}

\begin{remark}\label{rem:gactsoncomp}
Observe that if $(C, v)$ is $N$--feasible, then $(g\cdot C, g v)$ is $N$--feasible for all $g\in G$. Thus there is a bijection $f : \comp{v}{N}\to \comp{gv}{N}$ given by $C\mapsto g\cdot C$.
\end{remark}

Morally, if a pair $(C, v)$ is $N$--feasible, we should think that $v$ is in the ``interior'' of $C$ when viewed as a subset of $ \mb X \cup X$. Using the notion of a feasible pair, we introduce the $N$--core, which is a stable (in the sense of \cite{Durham-Taylor}) subset of the weak hull of a connected component $C$.

\begin{definition}[Core]\label{def_Core}
Let $N$ be a Morse gauge and $(C, v)$ be an $N$--feasible pair. The $N$--core of $(C, v)$, denoted by $\core{C, v}{N}$, is the set of vertices $w\in X$ such that $(C, w)$ is $N$--feasible and $[v, w]$ is $N$-Morse.
\end{definition}

Morally one should think of $\core{C, v}{N}$ as the set of vertices $w\in X$ which lie on an $N$-Morse ray starting at $v$ and ending in $C$. The vertex $v$ can be thought of as a chosen basepoint. Dual to the $N$--core, we can define the $N$--anti-core of a feasible pair $(C, v)$. While the $N$-core is a set of vertices (and hence a set of certain translates of $v$), the anti-core is a set of translates of $C$.

\begin{definition}[Anti-core]\label{def_hull}
Let $N$ be a Morse gauge and let $(C, v)$ be an $N$--feasible pair. The $N$--anti-core of $(C, v)$, denoted by $\hull{C, v}{N}$, is the set of translates $g\cdot C$ of $C$ satisfying that $(g\cdot C, v)$ is $N$--feasible and that $[v, gv]$ is $N$-Morse.
\end{definition}

\begin{remark}\label{remark:relation_Core_hull}   
    The $N$--anti-core and the $N$--core are related as follows. If $gv\in \core{C, v}{N}$ for some $g\in G$, then $g^{-1}\cdot C \in \hull{C, v}{N}$.   If $g\cdot C\in \hull{C, v}{N}$, then there exist $h\in G$ such that $h\cdot C = g\cdot C$ and $h^{-1}v\in \core{C, v}{N}$. However, it does not follow that $g^{-1}v\in \core{C, v}{N}$.
\end{remark}

\textbf{Notation:} We denote the set of non-singleton connected components of the Morse boundary $\mb G$ by $\mc C_G$.

\begin{lemma}\label{lemma:point-to-finite-compint}
    Let $\mc A\subset \mc C_G$ be a set of connected components of $\mb G$. If $\mc A$ is point-convergent, then $\comp{v}{N}\cap \mc A$ is finite for all vertices $v\in X$ and all Morse gauges $N$. 
\end{lemma}
\begin{proof}
     Assume that there exists a vertex $v\in X$ and a Morse gauge $N$ for which $\comp{v}{N}\cap \mc A$ is infinite. We show that $\mc A$ is not point-convergent. Since $\comp{v}{N}\cap \mc A$ is infinite, there exist sequences $(C_n)_{n}$ of pairwise distinct connected components $C_n\in \mc A$ and $(\lambda_n)_n$ of $N$-Morse geodesic lines such that $(C_n, v)$ is $N$--feasible and $\lambda_n$ is a witness of $(C_n, v)$ for all $n\in \N$. Up to passing to a subsequence, we may assume that the lines $\lambda_n$ converge to a geodesic line $\lambda$ going through $v$. By \cref{lemma:convergence_of_N_morse}, $\lambda$ is $N$-Morse and the sequences $(\lambda_n^+)_n$ and $(\lambda_n^-)_n$ converge to $\lambda^+$ and $\lambda^-$ respectively. Consequently, both $\lambda^-$ and $\lambda^+$ lie in $\lim_{n\to \infty} C_n$, implying that $\mc A$ is not point-convergent.
\end{proof}

\begin{lemma}\label{lemma:orbit-convergence_implies_finite_hull}
    Let $N$ be a Morse gauge and let $(C, v)$ be an $N$--feasible pair. If $\{g\cdot C\}_{g\in G}$ is point-convergent, then $\hull{C, v}{M}$ is finite for all Morse gauges $M\geq N$.
\end{lemma}

\begin{proof}
    The set $\hull{C, v}{M}$ is a subset of $\comp{v}{M}\cap \{g\cdot C\}_{g\in G}$. Hence the statement follows from \cref{lemma:point-to-finite-compint}
\end{proof}

\begin{lemma}\label{lemma:full_conv}
    The following properties are equivalent: 
    \begin{enumerate}
        \item There exists a vertex $v\in X$ for which $\comp{v}{N}$ is finite for all Morse gauges $N$
        \item $\comp{v}{N}$ is finite for all vertices $v\in X$ and Morse gauges $N$.
        \item The family $\mc C_G$ of non-singleton connected components of $\mb G$ is point-convergent.
    \end{enumerate}
\end{lemma}

\begin{figure}[H]
    \centering
    \includegraphics[scale=0.7]{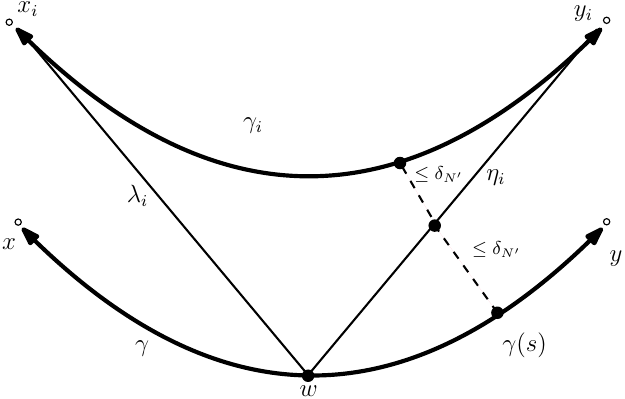}
    \caption{Proof of Lemma \ref{lemma:full_conv}}
    \label{fig:full_conv}
\end{figure}
\begin{proof}
    $(1)\implies (2):$ This follows directly from \cref{rem:gactsoncomp}.
    
    $(2) \implies (3):$ Assume that $\mc C_G$ is not point-convergent. We will show that there exists a vertex $v\in X$ and Morse gauge $N$ for which $\comp{v}{N}$ is infinite. Since $\mc C_G$ is not point-convergent, there exists a sequence $(C_i)_{i \in \mathbb N}$ of pairwise distinct connected components $C_i\in \mc C_G$ and two convergent sequences $(x_i)_{i \in \mathbb N}$ and $(y_i)_{i \in \mathbb N}$ of Morse directions satisfying $x_i,y_i \in C_i$ and whose limits $x=\lim_{i \to \infty}{x_i}$ and $y=\lim_{i \to \infty}{y_i}$ are distinct. Let $\gamma$ be a Morse geodesic line in $X$ with $\gamma^-= x$ and $\gamma^+=y$ and let $w = \gamma(0)$ be a point on $\gamma$. The sets $\{x_i\}_{i\in \N}\cup \{x\}$ and $\{y_i\}_{i\in \N}\cup \{y\}$ are compact. By \cref{lemma:compact_subsets}, compact sets are contained in the $N$-Morse boundary for some $N$. Implying that there exists a Morse gauge $N$ such that any ray starting at $w$ and ending in $x_i, y_i, x$ or $y$ is $N$-Morse. For each $i$ let $\gamma_i$ be a Morse geodesic line with endpoints $\gamma_i^- = x_i$ and $\gamma_i^+ = y_i$ and let $\lambda_i$ and $\eta_i$ be Morse geodesic rays starting at $w$ and ending in $x_i$ and $y_i$ respectively. This is depicted in Figure~\ref{fig:full_conv}. The lines $\eta_i, \lambda_i$ and $\gamma$ are $N$-Morse and by \cref{lemma:triangles} about Morse triangles, the lines $\gamma_i$ are $N' = f(N)$-Morse for all $i$. 

    We now try to find a point $v$ which infinitely many of the $\gamma_i$ pass through.
    Let $w' = \gamma(s)$ be a point on $\gamma$ with $s = 10 \delta_{N'}$, the point $v$ will be in the $2\delta_{N'}$--neighbourhood of $w'$. Up to passing to a subsequence, the $\lambda_i$ converge (uniformly on compact sets) to a geodesic ray $\lambda$, which is $N$-Morse by \cref{lemma:convergence_of_N_morse} and satisfies $\lambda^+ = \gamma^- = x$. Analogously, the $\eta_i$ converge to an $N$-Morse geodesic ray $\eta$ with $\eta^+ = y = \gamma^+$. 
    
    Consequently, there exists $i_0$ such that for all $i > i_0$, $d(\lambda_i, \gamma(s)) > 3\delta_{N'}$ and there exists a point $p_i$ on $\eta_i$ such that $d(p_i, \gamma(s))\leq \delta_{N'}$. By \cref{lemma:triangles} about Morse triangles, for all $i> i_0$, we have that $d(\gamma_i, p_i)\leq \delta_{N'}$ or $d(\lambda_i, p_i)<\delta_{N'}$. The latter contradicts $d(\lambda_i, \gamma(s)) > 3\delta_{N'}$. Thus, $d(\gamma_i, p_i)\leq\delta_{N'}$ implying $d(\gamma_i, \gamma(s))\leq2\delta_{N'}$. Since the graph $X$ is proper, by potentially passing to a subsequence, we can assume that there exists a vertex $v\in X$ in the $2\delta_{N'}$-neighbourhood of $\gamma(s)$ such that all lines $\gamma_i$ for $i> i_0$ pass through $v$. Consequently, $C_i\in \comp{v}{N'}$ for all $i> i_0$, which implies that $\comp{v}{N'}$ is infinite and hence concludes the proof.

    $(3)\implies (1):$ This follows from \cref{lemma:point-to-finite-compint}.
\end{proof}

\section{Stabilisers of connected components with finite anti-core}\label{sec:stabiliser}

In this section we prove \cref{thm:stabilizer}, which states that if $\hull{C, v}{M}$ is finite for all $M$, then $C$ is the relative Morse boundary of its stabiliser. Together with \cref{lemma:orbit-convergence_implies_finite_hull}, this implies \cref{thm:main}, which states that if the orbit of a non-singleton connected component $C$ is point-convergent, then the relative Morse boundary of its stabiliser $H$ is the connected component $C$. 

There are two key steps in the proof of \cref{thm:stabilizer}. The first one is \cref{lemma:close_to_stab}, which uses the relation between the $N$--core and $N$--anti-core to show that finiteness of the $N$--anti-core implies that the $N$--core is contained in a neighbourhood of the stabiliser of $C$. The second one is \cref{lemma:close_to_C}, which states that $(\mb H, G)\subset C$. The main proof strategy for \cref{lemma:close_to_C} is to show that for every ray $\gamma$ with endpoint in $(\mb H, G)$, there exists a sequence of rays whose endpoints lie in $C$ and converge to $\gamma^+$.

We start by showing that if a Morse ray $\gamma$ ends in $C$, then it is contained in the neighbourhood of the $M'$--core of $(C, v)$ for some Morse gauge $M'$.

\begin{lemma}\label{lemma:close_to_Core}
    Let $N$ be a Morse gauge and let $(C, v)$ be an $N$-feasible pair. Let $M\geq N$ be a Morse gauge and let $\gamma$ be an $M$-Morse geodesic ray starting at $v$ and whose endpoint $\gamma^+$ is in $C$. There exists a Morse gauge $M'$ and constant $D$ only depending on $M$ such that $\gamma\subset {\mc N}_D \left(\core{C, v}{M'}\right)$.
\end{lemma}

\begin{figure}[H]
    \centering
    \includegraphics[scale=0.7]{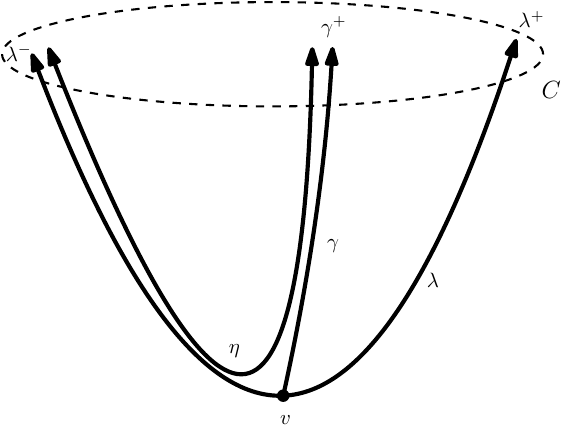}
    \caption{Proof of \cref{lemma:close_to_Core}}
    \label{fig:lemma4_1}
\end{figure}

\begin{proof}
Let $\lambda$ be an $N$-Morse witness of $(C, v)$. By potentially replacing $\lambda$ by its inverse, we may assume that $\lambda^-\neq \gamma^+$. Every point $x$ on $\lambda$ satisfies that $[x, v]$ is $N$-Morse and hence $x\in \core{C, v}{N}$. 

Let $\eta$ be a Morse geodesic line satisfying $\eta^- = \lambda^-$ and $\eta^+ = \gamma^+$. This is depicted in Figure \ref{fig:lemma4_1}. \cref{lemma:triangles} about Morse triangles shows that $\eta$ is $M'' = f(M)$-Morse and any point $w$ on $\eta$ satisfies that $[v, w]$ is $M' = f(M'')$-Morse.
Thus, any point $w$ on $\eta$  satisfies $w\in \core{C, v}{M'}$. \cref{lemma:triangles} about Morse triangles shows that $\gamma$ is in the $D = \delta_{M'}$-neighbourhood of $\eta\cup\lambda$ and hence in the $D$-neighbourhood of $\core{C, v}{M'}$. Since $D$ and $M'$ only depend on $M$, the result follows.
\end{proof}

\begin{lemma}\label{lemma:close_to_stab}
    Let $N$ be a Morse gauge and let $(C, v)$ be an $N$-feasible pair. Let $H\leq G$ be the stabiliser of $C$. If $\hull{C, v}{N}$ is finite, then there exists a constant $k> 0$ such that $\core{C, v}{N} \subseteq {\mc N}_k ( H)$.
\end{lemma}
\begin{proof}
    Assume that $\hull{C, v}{N}$ is finite. That is, $\hull{C, v}{N} = \{g_1\cdot C, \ldots g_n\cdot C\}$ for some elements $g_i\in G$. We show that the statement holds for 
    \[k = \max_{1\leq i\leq n}\{d (v, g_i)\}.\]

    Let $g v\in \core{C, v}{N}$. By \cref{remark:relation_Core_hull}, $g^{-1}\cdot C \in\hull{C, v}{N}$, which implies that $gg_i\in H$ for some $1\leq i \leq n$. Consequently, $d(g v, H)\leq d(gv, gg_i)\leq  d(v, g_i)\leq k$, which concludes the proof.
\end{proof}

\cref{lemma:close_to_Core} and \cref{lemma:close_to_stab} show that if $\hull{C, v}{N}$ is finite for all Morse gauge $M\geq N$, then $C\subset (\mb H, G)$. Indeed, if a Morse geodesic ray $\gamma$ starts in $v$ and ends in $C$, \cref{lemma:close_to_Core} shows that it is contained in a neighbourhood of $\core{C, v}{M}$ for some Morse gauge $M$. Then \cref{lemma:close_to_stab} shows that $\gamma$ is contained in a neighbourhood of $H$, implying that $\gamma^+\in (\mb H, G)$. It remains to show that $(\mb H, G)\subset C$, which is the content of the following lemma.

\begin{lemma} \label{lemma:close_to_C}
Let $N$ be a Morse gauge and let $(C, v)$ be an $N$-feasible pair. Let $H$ be the stabiliser of $C$. We have that $(\mb H, G)\subset C$.    
\end{lemma}

\begin{figure}[H]
    \centering
    \includegraphics[scale=0.7]{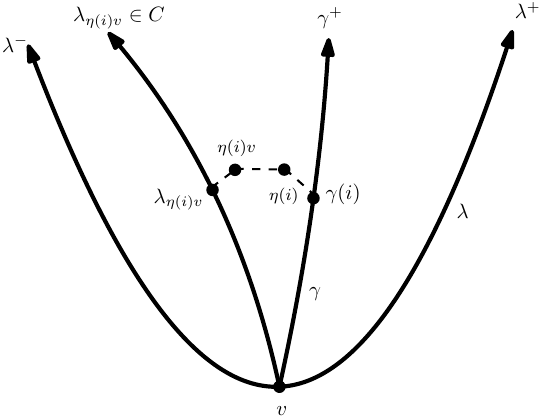}
    \caption{Proof of \cref{lemma:close_to_C}}
    \label{fig:lemma4_3}
\end{figure}
\begin{proof}
Let $\lambda$ be an $N$-Morse witness of the pair $(C, v)$. Let $\gamma$ be an $M$-Morse geodesic ray starting at $v$ and with $\gamma^+\in (\mb H, G)$. Let $k$ be a constant such that $\gamma\subset \mc N_k(H)$. For each integer $i\geq 0$, let $\eta(i)$ be an element of $H$ at distance at most $k$ from $\gamma(i)$. This is depicted in Figure \ref{fig:lemma4_3}.

Since, $d(\gamma(i), \eta(i))\leq k$ we have that the endpoints of $[v, \eta(i)v]$ are in the $d(1, v) + k$-neighbourhood of of $\gamma$. Thus by \cref{lem_neibMorse}, $[v, \eta(i)v]$ is $M'$--Morse, where $M'$ only depends on $d(1, v)$, $k$ and $M$.

Consider the ray $\lambda_{\eta(i)v}$ starting at $v$ as defined in \cref{lemma:corresponding_ray}.
We have $\lambda_{\eta(i)v}^+\in \{\eta(i)\cdot \lambda^{-}, \eta(i)\cdot \lambda^{+}\}$. Since $\eta(i)\in H$, it stabilises $C$. Hence we have that $\lambda_{\eta(i)v}^+ \in C$. 

Furthermore, by \cref{lemma:corresponding_ray} and the triangle inequality, 
\begin{align*}
    d(\gamma(i), \lambda_{\eta(i)v})&\leq d(\gamma(i), \eta(i)) + d(\eta(i), \eta(i)v) + d(\eta(i)v, \lambda_{\eta(i)v})\\
    &\leq k + d(1, v) + D,
\end{align*}
where $D$ is the constant from \cref{lemma:corresponding_ray} and only depends on $M'$.  Thus, for example by \cite[Lemma A.6]{Cordes}, the sequence $(\lambda_{\eta(i)v}^+)_i$ converges to $\gamma^+$. Since connected components are closed and $\lambda_{\eta(i)v}^+\in C$ for all $i$, we have that $\gamma^+\in C$.
\end{proof}

\begin{prop}\label{thm:stabilizer}
     Let $C$ be a non-singleton connected component and $H$ its stabiliser. If there exists a vertex $v$ and a Morse gauge $N$ such that $(C, v)$ is $N$-feasible and for every Morse gauge $M\geq N$ $\hull{C, v}{M}$ is finite, then $(\mb H,  G) = C$. 
\end{prop}
\begin{proof}
    \cref{lemma:close_to_C} shows that $(\mb H, G)\subset C$. On the other hand, if $\gamma$ is a Morse ray with $\gamma^+\in C$ and starting at $v$, then \cref{lemma:close_to_Core} shows that $\gamma$ is at bounded distance form the $M$--core of $(C, v)$ for some $M$. By \cref{lemma:close_to_stab}, the $M$--core of $(C, v)$ is at bounded distance from $H$ and hence $\gamma^+\in (\mb H, G)$.
\end{proof}

We are now ready to prove \cref{thm:main}.

\begin{proof}[Proof of \cref{thm:main}] Let $C$ be a non-singleton connected component of $\mb G$ whose orbit is point-convergent.
    Since $C$ is not a singleton, there exists a Morse gauge $N$ and a vertex $v\in X$ such that $(C, v)$ is $N$--feasible. By \cref{lemma:orbit-convergence_implies_finite_hull}, $\hull{C, v}{M}$ is finite for all $M\geq N$. \cref{thm:stabilizer} concludes the proof.
\end{proof}

\section{Graph of groups whose family of non-singleton components is point-convergent}\label{sec:graph_of_groups}
In this section we prove \cref{cor:graph_of_groups}, which states that the set of non-singleton connected components $\mc C_G$ is point-convergent for fundamental groups $G$ of certain graphs. These groups include a large class of groups that are Morse-accessible. As this is a purely topological condition, it is preserved under quasi-isometries. Thus any group quasi-isometric to $G$ satisfies this property too. Accordingly, this section proves \cref{cor:q_i_to_amalgam}.

\subsection{Setup and Notation}
\label{subsec}
In this subsection, we setup the notation for the proof of \cref{cor:graph_of_groups}. Apart from slight changes (most prominently the definition of the relative Morse boundary), we work with the notation of \cite{FK22} and more details can be found there.
In this subsection, we introduce basics on graph of groups and notation. As we study the same setting as in \cite{FK22}, we use the same notation as in \cite{FK22} and cite the lemmas from  \cite{FK22} we need.

Let $G$ be a finitely generated group group that splits as graph of groups. By \cite{Serre2012trees} this is equivalent to the existence of a simplicial tree on which $G$ acts without inversions and without global fixed point. The latter means that the the action is \textit{non-elliptic}. 

\begin{Ass}
We assume that
    \begin{itemize}
    \item $G$ is generated by a finite subset $S_G\cu G$ with $1\in S_G$ and $S_G=S_G^{-1}$;
    \item $\G$ denotes the corresponding Cayley graph of $G$, endowed with its graph metric $d_{\G}$;
    \item there is a non-elliptic, minimal action without inversions on a simplicial tree $G\acts T$ with edge set $\E(T)$;
    \item all edge-stabilisers $E\leq G$ are finitely generated and undistorted in $G$, and satisfy $(\mb E,G)=\emptyset$.
\end{itemize}
\end{Ass}

  We refer to the vertex stabilisers and edge stabilisers for the action $G\acts T$ as \emph{vertex groups} and \emph{edge groups} respectively and usually denote them by $V$ and $E$ respectively. To avoid confusion between  edges and vertices of $\G$ and $T$, we will denote the latter by fraktur letters $\mf{e}$ and $\mf{x}$.

 \begin{remark}\label{remark:different-relative-Morse-boundary}
     In light of \cref{lemma:defintions_agree}, the relative Morse boundary (as defined in this paper) of undistorted subgroups is empty if and only if the relative Morse boundary as defined in \cite{FK22} is empty. Thus the groups $G$ as defined above satisfy the assumptions of groups studied in \cite{FK22}.
 \end{remark}

\begin{remark}\label{finite graph of groups}
   As in \cite[Remark 2.13]{FK22} we remark that the action $G\acts T$ is cocompact. Moreover, if all edge groups are finitely generated and undistorted in $G$, then all vertex groups are finitely generated and undistorted in $G$. A proof for this well-known fact can be found in \cite[Lemma 2.15]{FK22}.
\end{remark}

Every edge $\mf{e}\in\E(T)$ corresponds to two \textit{halfspaces} $\mf{h},\mf{h}^*$ that are the two connected components of  the graph we obtain by removing the inner of the edge $\mf{e}$ from $x$.  We always denote the complementary halfspace of a halfspace $\mf{h}$ by $\mf{h}^*$. Let $\H(T)$ be the set of all halfspaces of $T$. If $\mf h$ is a halfspace in $\H(T)$, we denote by $G_{\mf{h}}\leq G$ its stabiliser in $G$.
 Since $G$ acts without inversions, $G_{\mf{h}}$ coincides with the stabiliser $G_{\mf{e}}$ of the edge $\mf{e}$ associated to $\mf{h}$. Thus, if $p\in T^{(0)}$ is a basepoint and $f_p\colon G\ra T$ the corresponding orbit map $f_p(g)=gp$, then every $\mf{h}\in\H(T)$ corresponds to a $G_{\mf{h}}$--invariant partition $G=f_{\mf p}^{-1}(\mf{h})\sqcup f_{\mf p}^{-1}(\mf{h}^*)$. These properties make it possible to define for every edge $\mf e \in \mc E(T)$ a corresponding subgraph $\G(\mf{e})$ that reflects the action $G \acts T$. This is the content of the following lemma from \cite{FK22}.

As in \cite{FK22}, we denote closed metric balls by $\mc{B}(x,r)$ and closed metric neighbourhoods of subsets by $\mc{N}(A,r)$. Where necessary, we may add a subscript $\mc{B}_X(x,r)$ or $\mc{N}_Y(A,r)$ to specify the relevant space.
\begin{lemma}\label{AIS lem}[Lemma 3.2 of \cite{FK22}]
    For every edge $\mf{e}\in\E(T)$ we can choose a subgraph $\G(\mf{e})\cu\G$ so that the following hold:
    \begin{enumerate}
        \item $\G(\mf{e})$ is connected, $G_{\mf{e}}$--invariant and $G_{\mf{e}}$--cocompact;
        \item $\G(g\mf{e})=g\G(\mf{e})$ for all $g\in G$ and $\mf{e}\in\E(T)$;
        \item if $\mf{h},\mf{h}^*$ are the two halfspaces determined by $\mf{e}$, then $\G(\mf{e})$ contains every edge of $\G$ with an endpoint in $f_{\mf p}^{-1}(\mf{h})$ and the other in $f_{\mf p}^{-1}(\mf{h}^*)$;
        \item for every $r\geq 0$, we have $\mc{N}_{\G}(f_{\mf p}^{-1}(\mf{h}), r)\cap\mc{N}_{\G}(f_{\mf p}^{-1}(\mf{h}^*), r)\cu\mc{N}_{\G}(\G(\mf{e}),r)$;\label{ais4}
        \item each edge $e\cu\G$ lies in the subgraph $\G(\mf{e})$ for at most finitely many edges $\mf{e}\in\E(T)$.
    \end{enumerate}
\end{lemma}

\begin{remark} \label{rem:X(e)qi}
   The subgraph $X(\mf e)$ in \cref{AIS lem} is constructed in \cite{FK22} so that it contains the vertex set $G_{\mf{e}}$. Note that $G_{\mf{e}}$ does not only act cocompactly by isometries on $X(\mf e)$ but also properly~\cite{hyperbolic_spaces}[Ch. I.8, Prop 8.5(4)]. So, $G_e$ acts geometrically on $X(\mf e)$ and $X(\mf e)$ is quasi-isometric to $G_{\mf e}$. 
\end{remark}

Let $\mf{e}$ be an edge in $T$ and $X(\mf{e})$ be the subgraph of $X$ as in \cref{AIS lem}. Let $\mf{h}$ be one of the two halfspaces determined by $\mf e$. Then $X(\mf h)$ denotes the subspace of $X$ that is obtained from the vertex set $f_p^{-1}(\mf h)\setminus X(\mf e)$ by adding all (half-open) edges of $X \setminus X(\mf e)$ with an endpoint in $f_p^{-1}(\mf h)\setminus X(\mf e)$. This way, we obtain a $G_{\mf e}$-invariant partition of $X$: 
\[X=X(\mf h) \cup X(\mf e)\cup X(\mf h*)\]
Recall that we assume that $(\mb E,G)=\emptyset$ for every edge group $E$. This assumption is not needed so far but becomes important now. 
If $x\in\mb G$, we consider the following subset of $\H(T)$:
\begin{align*}
\s(x)\coloneqq\{\mf{h}\in\H(T) \mid \text{if $x=[\g]$, then $\g$ is eventually contained in $\G(\mf{h})$}\}.
\end{align*}
Using that the relative Morse boundary of every edge group is empty, it is proven in \cite{FK22}  that this is well-defined and an \emph{ultrafilter}, i.e. it satisfies the following two properties: 
\begin{enumerate}
\item if $\mf{e}\in\E(T)$, then exactly one of the two halfspaces determined by $\mf{e}$ lies in $\s(x)$;
\item if $\mf{h}_1,\mf{h}_2\in\s(x)$, then $\mf{h}_1\cap\mf{h}_2\neq\emptyset$.
\end{enumerate}

\begin{remark}[Remark 3.6 of \cite{FK22}]\label{same ultrafilter rmk}
If $x,y\in\mb G$ are in the same connected component, then $\s(x)=\s(y)$. 
\end{remark}

For every vertex $\mf x\in T^{(0)}$, the set $\s_{\mf x}\coloneqq \{\mf{h}\in\H(T) \mid \mf x\in\mf{h}\}$
is an ultrafilter that plays an important role in the following lemma and the proof of \cref{lemma:finitely_many_vert_groups}.
Every ultrafilter $\s\cu\H(T)$ that does not contain infinite descending chains of halfspaces is of the form $\s_{\mf x}$ for a vertex $\mf x\in T^{(0)}$.

\begin{lemma}[Lemma 3.7 of \cite{FK22}]\label{line ultrafilter}
Let $\alpha$ be an $N$-Morse geodesic line whose endpoints at infinity $\alpha^{\pm}\in\mb G$ lie in the same connected component of $\mb G$. Set $\s\coloneqq\s(\alpha^+)$, which coincides with $\s(\alpha^-)$ by \cref{same ultrafilter rmk}. Then, the following hold:
\begin{enumerate}
\item there is a constant $D(N)$ depending only on $N$ so that $\alpha$  lies in the intersection of the $D(N)$-neighbourhoods of the subsets $\G(\mf{h})\cu\G$ with $\mf{h}\in\s$;\label{line1}
\item there exists a vertex ${\mf x}\in T$ such that $\s=\s_{\mf x}$;
\item $\alpha$ stays at bounded distance from the stabiliser $G_{\mf x}\leq G$.
\end{enumerate}
\end{lemma}

\subsection{Proof of \cref{cor:graph_of_groups}}
With notation as in \cref{subsec}, we prove \cref{cor:graph_of_groups}. It is a consequence of the following lemma.

\begin{lemma}\label{lemma:finitely_many_vert_groups}
    Let a finitely generated group $G$ split as a graph of groups. Suppose that
\begin{itemize}
    \item all edge groups are finitely generated and undistorted in $G$;
    \item $(\partial_*E,G)=\emptyset$ for every edge group $E$.
\end{itemize}
   Let $N$ be a Morse gauge and let $v$ be a vertex in $X$. Let $(\gamma_n)$ be a sequence of converging $N$-Morse geodesic lines which contain $v$ and whose endpoints $\gamma_n^- $ and $\gamma_n^+$ lie in a connected component $C_n\subset \mb G$. Then there exists $n_0$ and a vertex $\mf x\in T$ such that for all $n\geq n_0$, $\sigma(\gamma^+) = \sigma_{\mf x}$.
\end{lemma}

\begin{proof}
    Let $\gamma$ be the geodesic line that the $\gamma_n$ converge to. By \cref{lemma:convergence_of_N_morse}, $\gamma$ is $N$-Morse, and we have that $\gamma^- = \lim_{n\to \infty} \gamma_n^-$, $\gamma^+ = \lim_{n\to \infty} \gamma_n^+$. In other words, $\gamma^-, \gamma^+$ both lie in $\lim_{n\to \infty} C_n$, which is connected since it is a limit of connected sets. Let $C$ be the connected component containing $\gamma^-$ and $\gamma^+$. 
    Let $\mf x\in T$ be the vertex such that $\sigma_{\mf x} = \sigma(\gamma^+)$.

    Assume that for some $n$ we have that $\sigma(\gamma_n^+) = \sigma_{\mf z}$ for some vertex $\mf z\in T$ with $\mf z\neq \mf x$. Let $\mf e$ be an edge in $T$ on the geodesic from $\mf x$ to $\mf z$. Let $\mf h$ and $\mf h^*$ be the pair of half-spaces corresponding to $\mf e$ containing $\mf x$ and $\mf z$ respectively. Thus, $\mf h\in \sigma_{\mf{x}}$ and $\mf h^*\in \sigma_{\mf{z}}$. By \cref{line ultrafilter}\eqref{line1}, $\gamma$ and $\gamma_n$ are contained in the $D(N)$-neighbourhood of the subset $X(\mf h)$ and  $X(\mf h^*)$ respectively. For any $w \in \mc{N}_X(\gamma, 1)\cap \mc{N}_X(\gamma_n, 1)$, we have that $w \in \mathcal{N}_X(X(\mf h),D(N)+1) \cap \mathcal{N}_X(X(\mf h^*),D(N)+1) $. 
    \cref{AIS lem}\eqref{ais4} implies that  $w \in \mc N_X(X(\mf e), D(N)+1)$. Since $v$ lies on all $\gamma_n$, we have in particular that $v\in \mc N_X(X(\mf e), D(N)+1)$.
    
    We observe: By~\cref{AIS lem}, for each edge $e$ in $X$ there are only finitely many edges $\mf e \in \mathcal E(T)$ with $e \in X(\mf e)$. As $X$ is locally finite, there are only finitely many edges $\mf e \in \mathcal E(T)$ so that $d_X(X(\mf e), v)\leq D(N)+1$.

    If there are infinitely many $n$ such that $\sigma(\gamma_n^+)\neq \sigma(\gamma^+)$, the observation above implies that there exists an edge $\mf e\in \mathcal E(T)$ such that 
    \[
    \mc N_X(\gamma, 1)\cap \mc N_X(\gamma_n, 1)\subset \mc N_X(X(\mf e), D(N)+1)
    \]
     for infinitely many $n$. Since the $\gamma_n$ converge to $\gamma$, we have $\gamma\subset \mc N_X(X(\mf e), D(N)+1)$. By \cref{rem:X(e)qi}, $X(\mf e)$ contains the vertex set of $G_{\mf e}$ and is quasi-isometric to $G_{\mf e}$. Thus, $\gamma$ lies in  $(\mb G_{\mf e}, G)$ -- a contradiction to $(\mb G_{\mf e}, G)$ being empty.

    Thus there are only finitely many $n$ such that $\sigma(\gamma)\neq \sigma(\gamma_n)$, which concludes the proof.
\end{proof}

We are now ready to prove \cref{cor:graph_of_groups}.

\begin{proof}[Proof of \cref{cor:graph_of_groups}]
    Let $v\in X$ be a vertex and let $N$ be a Morse gauge. By \cref{lemma:full_conv}, it suffices to show that $\comp{v}{N}$ is finite. Assume it is not. That is, there exists a sequence of $N$-Morse geodesic lines $(\gamma_n)_n$ going through $v$ and a sequence $(C_n)_n$ of distinct connected components $C_n\subset \mb G$ such that $\gamma_n^+, \gamma_n^-\in C_n$. Up to passing to a subsequence, we can assume that $(\gamma_n)_n$ is a sequence converging to an $N$-Morse geodesic line $\gamma$ with endpoints $\gamma^- = \lim_{n\to \infty} \gamma_n^-$ and $\gamma^+ = \lim_{n\to \infty} \gamma_n^+$ as $X$ is locally finite. Further, there exists a connected component $C$ such that $\gamma^-, \gamma^+\in C$  because for each $n$, $\gamma_n^+, \gamma_n^-$ lie in a connected component.

    By \cref{lemma:finitely_many_vert_groups} there exists $n_0$ and a vertex $\mf x\in T$ such that for all $n\geq n_0$ we have that $\sigma(\gamma_n^+) = \sigma_{\mf x}$. By \cref{line ultrafilter}(3), for all $n\geq n_0$ we have that $C_n\subset (\mb G_{\mf x}, G)$. 

    By the assumptions on $G$, one of the following holds

    \textbf{Case 1:} $(\mb G_{\mf x}, G)$ is connected. This contradicts the assumption that the connected components $C_n$ are pairwise distinct.

    \textbf{Case 2:}  $\iota\colon \mb G_{\mf x}\emb \mb G$ is a topological embedding and the set of connected components $\mc C_{G_{\mf x}}$ is point-convergent. In this case $\iota$ induces a bijection between the connected components contained in  $(\mb G_{\mf x}, G)$ and the connected components of $\mb G_{\mf x}$. In particular, $(\iota^{-1}(C_n))_{n\geq n_0}$ is a sequence of distinct connected components whose limit contains both $\iota^{-1}(\gamma^-)$ and $\iota^{-1}(\gamma^+)$. This is a contradiction to $\mc C_{G_{\mf x}}$ being point-convergent.

    Since there is a contradiction in both cases, $\comp{v}{N}$ cannot be infinite, which concludes the proof.
\end{proof}


\bibliographystyle{alpha}
\bibliography{bibliothek}

\end{document}